\documentclass[12pt]{extarticle}
\usepackage{amsmath, amsthm, amssymb, hyperref, color}
\usepackage[shortlabels]{enumitem}
\usepackage{graphicx}
\usepackage[all]{xypic}
\usepackage{makecell}
\usepackage[final]{pdfpages}
\setboolean{@twoside}{false}
\usepackage{pdfpages}
\usepackage{caption}
\usepackage{subcaption}
\usepackage{scalefnt}
\usepackage{verbatim}
\oddsidemargin \evensidemargin
\newtheorem{theorem}{Theorem}
\numberwithin{theorem}{section}
\newtheorem{proposition}[theorem]{Proposition}
\newtheorem{lemma}[theorem]{Lemma}
\newtheorem{corollary}[theorem]{Corollary}

\newtheorem{remark}[theorem]{Remark}
\newtheorem{example}[theorem]{Example}

\newcommand{\RR}{\mathbb{R}}

\newcommand{\PP}{\mathbb{P}}
\newcommand{\CC}{\mathbb{C}}

 \date{}

\title{\textbf{Real Rank Two Geometry}}

\author{Anna Seigal and Bernd Sturmfels}

\begin{document}

\maketitle

\begin{abstract} \noindent
The real rank two locus of an algebraic variety is the closure of the union of all secant lines spanned~by real points.
We seek a  semi-algebraic description of this set.
Its algebraic boundary consists of the
tangential variety and the
edge variety.  Our study of Segre and Veronese 
varieties yields a characterization of tensors of real rank~two.
\end{abstract}

\section{Introduction} \label{1}

Low-rank approximation of tensors is a fundamental problem in applied mathematics \cite{SL, KB}.
We here approach this problem from the perspective
of real algebraic geometry. Our  goal is to 
give an exact semi-algebraic description of
the set of tensors of real rank two and to characterize its  boundary.
This complements the results on tensors of
non-negative rank two presented in \cite{ARSZ}, and
it offers a generalization to the setting of
arbitrary varieties, following~\cite{BS}.

A familiar example is that of
$2 \times 2 \times 2$-tensors $(x_{ijk})$ with real entries.
Such a tensor lies in the closure of the real rank two tensors if and only if
the {\em hyperdeterminant} is non-negative:
\begin{equation}
\label{eq:hyperdet}
\begin{matrix}
 x_{000}^2 x_{111}^2
+x_{001}^2 x_{110}^2
+x_{010}^2 x_{101}^2
+x_{011}^2 x_{100}^2
+4 x_{000} x_{011} x_{101} x_{110}
+4 x_{001} x_{010} x_{100} x_{111}  \\
-2 x_{000} x_{001} x_{110} x_{111}
-2 x_{000} x_{010} x_{101} x_{111}
-2 x_{000} x_{011} x_{100} x_{111}\,\, \,\qquad  \qquad \\
-2 x_{001} x_{010} x_{101} x_{110}
-2 x_{001} x_{011} x_{100} x_{110}
-2 x_{010} x_{011} x_{100} x_{101} \,\,\, \geq \,\,\, 0.
\end{matrix}
\end{equation}
If this inequality does not hold then the tensor  has
rank two over $\mathbb{C}$ but rank  three over $\mathbb{R}$.

To understand this example geometrically,
consider the {\em Segre variety}
$X = {\rm Seg}(\PP^1 \times \PP^1 \times \PP^1)$,
i.e.~the set of rank one tensors,
regarded as points in the projective space $\PP^7 =
\PP( \CC^2 \otimes \CC^2 \otimes \CC^2)$.
The hyperdeterminant defines a quartic hypersurface $\tau(X)$ in $\PP^7$.
The real projective space $\PP^7_\RR$ is divided
into two connected components by its real points
$\tau(X)_\RR$. One of the two connected components is the locus
$\rho(X)$ that comprises the real rank two tensors.

\smallskip

This paper views real rank in a general geometric framework,
studied recently by Blekherman and Sinn~\cite{BS}.
Let $X$ be an irreducible variety in  a complex projective space $\PP^N$ that is defined over 
$\RR$ and whose set  $X_\RR = X \cap \PP^N_\RR$ of real points
 is Zariski dense in $X$. The {\em secant variety}  $\sigma(X)$
 is the closure of the set of points in $\PP^N$ that lie
 on a line spanned by two points in $X$. 
 The {\em tangential variety}  $\tau(X)$ 
 is a subvariety of the secant variety.
 Namely, $\tau(X)$  is the closure of the set of points in $\PP^N$ that
 lie on a tangent line to $X$ at a smooth point. In this paper closure is taken with respect to the Euclidean topology, unless otherwise specified. For the secant and tangential varieties above, the 
 Euclidean closure and Zariski closure coincide.
 
Our object of interest is the {\em real rank two locus} $\rho(X)$.
This is a semi-algebraic set in the real projective space $\PP^N_\RR$.
We define $\rho(X)$ as the (Euclidean) closure of the set of points 
that lie on a line spanned by two points in $X_\RR$.
Our hypotheses ensure that $\rho(X)$  is Zariski dense in $\sigma(X)$. 
The inclusion of the closed set
$\rho(X) $ in the real secant variety $ \sigma(X)_\RR$ is usually strict. The difference consists of
points of $X$-rank two whose real $X$-rank exceeds~two.

Two varieties most relevant for applications
are the {\em Segre variety} $X = {\rm Seg}(\PP^{n_1-1} \times \cdots \times \PP^{n_d-1})$ and
the {\em Veronese variety} $X = \nu_d ( \PP^{n-1} )$.
The ambient dimensions are $N = n_1 \cdots n_d -1$ and
$N = \binom{n+d-1}{d}-1$, and $X$ consists of (symmetric)
tensors of rank one. The secant variety $\sigma(X)$ is the closure of 
the set of tensors of complex rank two,
and  $\sigma(X)_\RR$ is the set of real points of that complex projective variety.
The real rank two locus $\rho(X)$ is the closure of the 
tensors of real rank two. This is a subset
of $\sigma(X)_\RR$. The containment is strict when $d \geq 3$.

It is instructive to examine the case of $ 3 \times 2 \times 2 $-tensors. 
The secant variety $\sigma(X)$ has dimension $9$ in $\PP^{11}$.
By \cite{Rai},
it  consists of all tensors whose $3 \times 4$ 
matrix flattening satisfies
\begin{equation}
\label{eq:threebyfour}
{\rm rank}
 \begin{pmatrix}
x_{000} & x_{001} & x_{010} & x_{011} \\
x_{100} & x_{101} & x_{110} & x_{111} \\
x_{200} & x_{201} & x_{210} & x_{211} 
\end{pmatrix} \quad \leq \quad 2.
\end{equation}
The tangential variety $\tau(X)$  has codimension one in $\sigma(X)$.
The ideal of $\tau(X)$ is generated by the  $3 {\times} 3$-minors
of (\ref{eq:threebyfour}) and six hyperdeterminantal
quartics \cite{Oed}.
The set difference $\sigma(X)_\RR \backslash \tau(X)_\RR$ is disconnected.
The closure of one of its connected components is
the real rank two locus $\rho(X)$. Theorem \ref{thm:subhyper} says that
$\rho(X)$ is defined by three inequalities like~(\ref{eq:hyperdet}).

\smallskip

This article makes the following contributions. In Section \ref{genX} we determine
the algebraic boundary of the real rank two locus $\rho(X)$,
and we characterize boundary points that can be selected by
Euclidean distance optimization.  These results (Theorems  \ref{thm:main}
and \ref{thm:notrank1}) are for general varieties $X$.
 Section  \ref{curves} offers a detailed study of the case when $X$ is a space curve.
Section \ref{sec:tensors} is devoted to the usual setting of tensors,
when $X$ is a Segre or Veronese variety. The real rank two locus for tensors is
 characterized by hyperdeterminantal inequalities (Theorem \ref{thm:subhyper})
and its algebraic boundary is given by the tangential variety (Theorem \ref{cor:segver}).
In Section~\ref{sec:five} we apply \cite{OR} to derive explicit equations
(in Corollary \ref{eq:quadrics}) for that  boundary when
 $X$ is the Veronese. We also
characterize symmetric $2 {\times} 2 {\times} \cdots {\times} 2 $-tensors of real rank two.

Our  results here lay the geometric foundations for a subsequent paper that  studies numerical
algorithms for finding best real border rank two approximations of a given tensor.

\section{Projective Varieties}\label{genX}

We fix an irreducible real projective variety $X \subset \PP^N$  whose set of real points $X_\RR$
 is Zariski dense in $X$. The  tangential variety  $\tau(X)$ is contained in the
 secant variety  $\sigma(X)$.  If the inclusion   $\tau(X) \subset \sigma(X)$
 is strict then both varieties have the expected dimensions:
 \begin{equation}
 \label{eq:expdim}
  {\rm dim}(\sigma(X)) = 2 \cdot {\rm dim}(X) +1 \qquad
 \hbox{and} \qquad
 {\rm dim}(\tau(X)) = 2 \cdot {\rm dim}(X) . 
\end{equation}
This is  Theorem 1.4 in Zak's book  \cite{Zak}.
  If  $\tau(X) = \sigma(X)$ then the variety $X$ is called
{\em defective}. Otherwise,  the equalities in (\ref{eq:expdim}) hold,
and we say that $X$ is {\em non-defective}.

We write $\hat X \subset \CC^{N+1}$ for the affine cone over $X$.
The {\em $X$-rank} of a vector $x$ in $\CC^{N+1}$  is the smallest
$r$ such that $x = x_1 + \cdots + x_r$ with $x_1,\ldots,x_r $ in $ \hat X$,
and analogously for points $x$ in $\PP^N$.
If $x$ is real then its {\em real $X$-rank} is the smallest $r$ such that $x = x_1 + \cdots + x_r$ with
$x_1,\ldots,x_r $ in $ {\hat X}_\RR$.
The loci of $X$-rank $\leq r$ and real $X$-rank $\leq r$ are typically not closed.
We define the {\em $X$-border rank} and {\em real $X$-border rank}
by passing to the closure of these loci.
The secant variety $\sigma(X)$ consists of points of $X$-border rank $\leq 2$.
The real rank two locus $\rho(X)$ consists of points of real $X$-border rank $\leq 2$.
The latter is Zariski dense in the former.

The {\em real rank two boundary}  $\partial(\rho(X))$ is the set
 $\rho(X)$ minus its relative interior. Here the term ``relative'' refers to 
 $\sigma(X)_\RR$ being the ambient topological space.
Note that  $\partial(\rho(X))$ and $\sigma(X)_\RR$
 are semi-algebraic subsets in $\PP^N_\RR$.
We also note that $\partial(\rho(X))$ equals the topological boundary of 
$\rho(X)$, as discussed for similar settings in \cite[\S 4]{LS} and \cite[\S 5]{MMSV}.
The Zariski closure of  the set $\partial(\rho(X))$ in $\PP^N$
is denoted $\partial_{\rm alg}(\rho(X))$ and is called the
{\em algebraic real rank two boundary} of $X$. This projective variety
has codimension one in $\sigma(X)$. Our aim is to describe~it.

We need the following definitions.  Let $p$ and $q$ be  distinct
smooth points on $X$ whose  corresponding tangent spaces 
$T_p(X)$ and $T_q(X)$ intersect in $\PP^N$.
The secant line spanned by such  $p$ and $q$ is called an 
{\em edge} of $X$. The Euclidean
closure of the union of all edges of $X$ is 
a Zariski closed subset in the complex projective space $\PP^N$. 
 This subset is the {\em edge variety} $\,\epsilon(X)$.
If ${\rm dim}(X) = (N-1)/2$ then the edge variety  $\epsilon(X)$ is usually a hypersurface
in  $\sigma(X)= \PP^N$. That hypersurface is the variety
$\, (X^{[2]})^*\,$ in  \cite{RS1}, where it
 plays~an important role in convex algebraic geometry. 
For curves $X$ in $\PP^3$, this is the {\em edge surface} studied in~\cite{RS2}.

\begin{theorem} \label{thm:main}
Let $X$ be a non-defective variety in $\PP^N$
whose real points are Zariski dense. If  the 
algebraic real rank two boundary of $X$ is non-empty then it is a variety
of pure codimension one inside the secant variety $\sigma(X)$. Its irreducible components
arise from the tangential variety and  the edge variety. In symbols,
we have the equi-dimensional inclusion
\begin{equation}
\label{eq:tauepsilon}
 \partial_{\rm alg}(\rho(X)) \,\,\subseteq \,\,\tau(X) \,\cup\,\epsilon(X).
\end{equation}
\end{theorem}

The hypothesis that $X$ is non-defective is essential for this theorem.
For instance, if $X$ is a plane curve in $\PP^2$ then
$X$ is defective. Blekherman and Sinn  \cite[\S 3]{BS} showed
that $ \partial_{\rm alg}(\rho(X))$ is a union of flex lines,
provided it is non-empty. Such flex lines are not covered by (\ref{eq:tauepsilon}).

\begin{remark}\label{rem:edge} \rm
The tangential variety $\tau(X)$ is always irreducible when $X$ is irreducible.
However, the edge variety $\epsilon(X)$ may be reducible even 
when $X$ is irreducible. For instance, this happens when 
$X$ is the elliptic curve obtained by intersecting two quadratic surfaces
in $\PP^3$; see \cite[Example 2.3]{RS2}. Therefore, it is
possible that  $\partial_{\rm alg}(\rho(X))$ has more than two irreducible components. 
By the definition of $\epsilon(X)$, any point on any irreducible component of $\epsilon(X)$
 is a limit of points lying on at least two secant lines through smooth points of $X$.
\end{remark}

\begin{proof}[Proof of Theorem \ref{thm:main}]
The fact that $\partial_{\rm alg}(\rho(X))$ is pure of dimension one
will be derived from the general result in \cite[Lemma 4.2]{Sinn}:
if a semialgebraic set $S \subset \RR^k$ is nonempty and contained in the
closure of its interior and the same is true for $\RR^k \backslash S$, then the algebraic
boundary of $S$ is a variety of pure codimension one. Since the property is local,
we can here replace $\RR^k$ by $X_\RR$. The argument below will show
that these hypotheses are satisfied here.

Recall from (\ref{eq:expdim}) that ${\rm dim}(\sigma(X)) =  2 \cdot{\rm dim}(X)+1$.
Hence, for a general real point $u$ on the secant variety $\sigma(X)$,
there are only finitely many pairs 
$\{v_1,w_1\}, \{v_2, w_2\}, \ldots,  \{v_k,w_k\}$ 
of points on $X$ such that the line spanned by
$v_i$ and $w_i$ contains $u$. The $2k$ non-singular points of $X$ can be
expressed locally as algebraic
functions of $u$, by the Implicit Function Theorem.
The point $u \in \sigma(X)_\RR$ lies in $\rho(X)$ if at least one of these pairs
$\{v_i,w_i\}$ consists of two real points, and it lies outside
$\rho(X)$ if none of the pairs $\{v_i,w_i\}$ are real. 
By our assumption that the left hand side of (\ref{eq:tauepsilon}) is non-empty, both cases are possible for $X$.

Consider a general real curve that passes through the
boundary $\partial(\rho(X))$ at a point $u^*$, and follow
the $k$ point pairs along that curve.
This uses the Curve Selection Lemma in Real Algebraic Geometry.
Precisely one of  two scenarios will happen at the transition point:

{\em Case 1}:  A pair $\{v_i,w_i\}$ of real points 
merges into a single point
on $X$ and then transitions to a pair of conjugate complex points.
As that transition occurs, the secant line degenerates to a tangent line. Hence the
corresponding point $u^*$ lies in the tangential variety $\tau(X)$.

{\em Case 2}:  Two real pairs $\{v_i,w_i\}$ and
$\{v_j,w_j\}$ come together, in the sense that
$v_i$ and $v_j$ converge to a point $v \in X$
while $w_i$ and $w_j$ converge to another point $w \in X$.
If this happens then the tangent spaces
$T_v(X)$ and $T_w(X)$ meet non-transversally, by the following argument. The secant lines through $u$ arising from the two pairs $\{ v_i, w_i\}$ and $\{v_j,w_j\}$ span a plane that contains the line from $v_i$ to $v_j$ and the line from $w_i$ to $w_j$. In the limit as $v_i,v_j \to v$,
$w_i,w_j \to w$
and $u \to u^*$, a line in $T_v(X)$ will be co-planar to a line in $T_w(X)$. The meeting point of the two lines is their non-transverse intersection. Hence the secant line spanned by $v$ and $w$ 
must be an edge. We conclude that $u^*$ lies in the edge variety $\epsilon(X)$.

Our argument above shows that a generic path through $\partial(\rho(X))$ meets the boundary at either $\tau(X)$ or $\epsilon(X)$. Since the set $\rho(X)$ does not have lower-dimensional components, the Zariski closure of such boundary points is the algebraic real rank two boundary $\partial_{\rm alg}(\rho(X))$. Since the two sets $\tau(X)$ and $\epsilon(X)$ are  Zariski-closed in $\PP^N$, it follows that $\partial_{\rm alg}(\rho(X))$ is contained in their union
 $\tau(X) \cup \epsilon(X)$.
\end{proof}

The present article was motivated by the following optimization problem:
$$
\hbox{Given data $u \in \RR^{N+1}$, find the point  $u^*$ in 
the real rank two locus
 $ \rho(X)$ that is closest to $u$.}
$$
Here and in what follows we identify real points on projective varieties, in $\PP_\RR^N$, and their affine cones
in $\RR^{N+1}$. The term
``closest'' refers to either the Euclidean norm or  a weighted Euclidean norm
as in \cite{DHOST, LS}.
 The algebraic complexity of this problem is measured by the
{\em Euclidean distance degree} (ED degree).
A priori, five scenarios govern the location of the closest approximation $u^*$ to a random data point $u$:
\begin{itemize}
\item[(a)] $u^*$ is the point
in $\sigma(X)_\RR$ that is closest to $u$, and it is a smooth point of $\sigma(X)$. \vspace{-0.1in}
\item[(b)] $u^*$ is the point in $X_\RR$ that is closest to $u$; in particular, it is a singular point of $\sigma(X)$.
\vspace{-0.1in}
\item[(c)] $u^*$ is the point in the singular locus of $\sigma(X)_\RR$ that is closest to $u$, but it is not in $X$.
\vspace{-0.1in}
\item[(d)] $u^*$ is the point in $\tau(X)_\RR$ that is closest to $u$.
\vspace{-0.1in}
\item[(e)] $u^*$ is the point in $\epsilon(X)_\RR$ that is closest to $u$.
\end{itemize}

The solutions $u^*$ in cases (d) and (e) are not critical for the
distance function on $\sigma(X)$.
The following theorem shows that case (b) cannot happen. This was proven
for tensors by Stegeman and Friedland \cite[Lemma 3.4]{StFr}.
We   generalize their result to arbitrary varieties.

\begin{theorem} \label{thm:notrank1}
Suppose that $\hat X$ does not lie on a hyperplane in $\RR^{N+1}$.
Let $u \in \mathbb{R}^{N+1}$ be a data point of real $\hat X$-border rank bigger than $r$ 
and $u^* \in \mathbb{R}^{N+1}$ its best approximation of real $\hat X$-border rank at most $r$.
 Then the real $\hat X$-border rank of $u^*$ 
is exactly $r$, not smaller.
\end{theorem}

The best approximation is taken with respect to a weighted Euclidean distance on $\RR^{N+1}$
where all weights are strictly positive.
The impossibility of case (b) is Theorem \ref{thm:notrank1} for $r=2$.

\begin{proof}
We begin with the case $r = 1$. Then $u \notin \hat X$ and we wish to show that its best border rank one approximation $u^*$ is non-zero. By our assumption, there exists a non-zero
vector $x $ in the affine cone 
$ \hat X$ that is not in the hyperplane perpendicular to $u$. 
This means that $\langle u, x \rangle \neq 0$, where the inner product comes from our choice of norm. The point $\,\frac{ \langle u, x \rangle }{ \langle x, x \rangle } x \,$ lies in 
 $\hat X$, and its squared distance to the given data point $u$ is
$$ \hspace{-26ex} \left| \left| \frac{ \langle u, x \rangle }{ \langle x, x \rangle } x - u \right| \right|^2
\,\, =\,\,\, \left\langle \frac{ \langle u,x \rangle } { \langle  x,x \rangle } x - u ,  
 \frac{ \langle u,x \rangle } { \langle  x,x \rangle } x - u \right\rangle $$
$$ \hspace{22ex} \,\, =  \,\, {\left( \frac{ \langle u,x \rangle}{ \langle x,x \rangle } \right)}^2 \langle x,x \rangle - 2  \frac{ \langle u,x \rangle }{ \langle x,x \rangle} \langle u,x \rangle + \langle u,u \rangle\,\, = 
\,\, \langle u,u \rangle - \frac{\langle u,x \rangle ^2}{ \langle x,x \rangle}. $$
This is strictly smaller than $ || u - 0 ||^2 = \langle u, u \rangle$, so the closest point to $u$ on $\hat X$ is non-zero.

We now suppose that $r \geq 2$ and let $u^*$ be the best approximation to $u$ among points of real 
$X$-border rank at most $r$.  We first suppose for contradiction that $u^*$ has real $\hat X$-border rank at most $r-1$. We then construct a strictly better border rank $r$ approximation of $u$ by combining $u^*$ with a best rank one approximation for $u - u^*$. 

The point $v = u - u^*$ is non-zero. Its best real $X$-rank one approximation $v^*$ is also non-zero. 
When $v \notin \hat X$, we use the first paragraph of the proof to see this; otherwise $v^* = v \neq 0$.
The point $u^* + v^*$ still has real $\hat X$-border rank at most $r$, and it is closer to $u$ than $u^*$, since
$$ || u - (u^* + v^* ) ||\, =\, || v - v^* ||\, < \,|| v - 0 ||\, = \,|| v || \,= \, || u - u ^* || .$$
Hence the best approximation to $u$ cannot have
real $\hat X$-border rank strictly less than $r$.
\end{proof}

We have shown that case (b) cannot happen for best approximation by $\rho(X)$.
All of the other four cases (a), (c), (d) and (e) are possible.
Case (a) is the usual best real rank two approximation and it occurs frequently.
Cases (d) and (e) occur for the curve in Example~\ref{spacecurve1}.
We close this section by showing that case (c) occurs for rank two tensor   approximation.

\begin{example} \rm
Let $N=26$ and  fix $X =  {\rm Seg}(\mathbb{P}^2 \times \mathbb{P}^2 \times \mathbb{P}^2 )$.
According to \cite[Cor.~7.17]{MOZ}, the singular locus of $\sigma(X)$ 
has three irreducible components, given by the three permutations of $\PP^2 \times \sigma({\rm Seg}(\PP^2 \times \PP^2))$. These
parametrize tensors $v \otimes M$, where $v \in \RR^3$ and $M$ is a $3 \times 3$-matrix of rank two.
Consider a data point $U = v \otimes M'$ where $M'$ is a general real $3 \times 3$-matrix.
Let $M$ be the best rank two approximation of $M'$. The entries of $v \otimes M$ are three copies of $M$, multiplied by coefficients $v_1$, $v_2$ and $v_3$. The tensor $U^* = v \otimes M$
gives the unique best approximation to $U$ in all three slices, hence $U^*$ is the best approximation to $U$ in $\rho(X)$.
\end{example}

\section{Space Curves} \label{curves}

Blekherman and Sinn \cite[\S 3]{BS} characterized the real rank two locus
$\rho(X)$ for a curve $X$ in the plane $\PP^2$.
In this section we examine the case when $X$
is a curve in $\PP^3$. We assume that $X$ does not
lie in a plane and that $X_\RR$ is Zariski dense in $X$.
The real $X$-rank of a general point $u \in \PP^3_\RR$ is either two or more,
depending on whether the plane curve obtained by
projecting $X$ from $u$ has real singularities or not.
Specifically, any node on the projected curve corresponds to a line
spanned by two real points on $X$ that passes through~$u$.

\begin{figure}[h]
  \begin{center}
  \includegraphics[height=6.2cm]{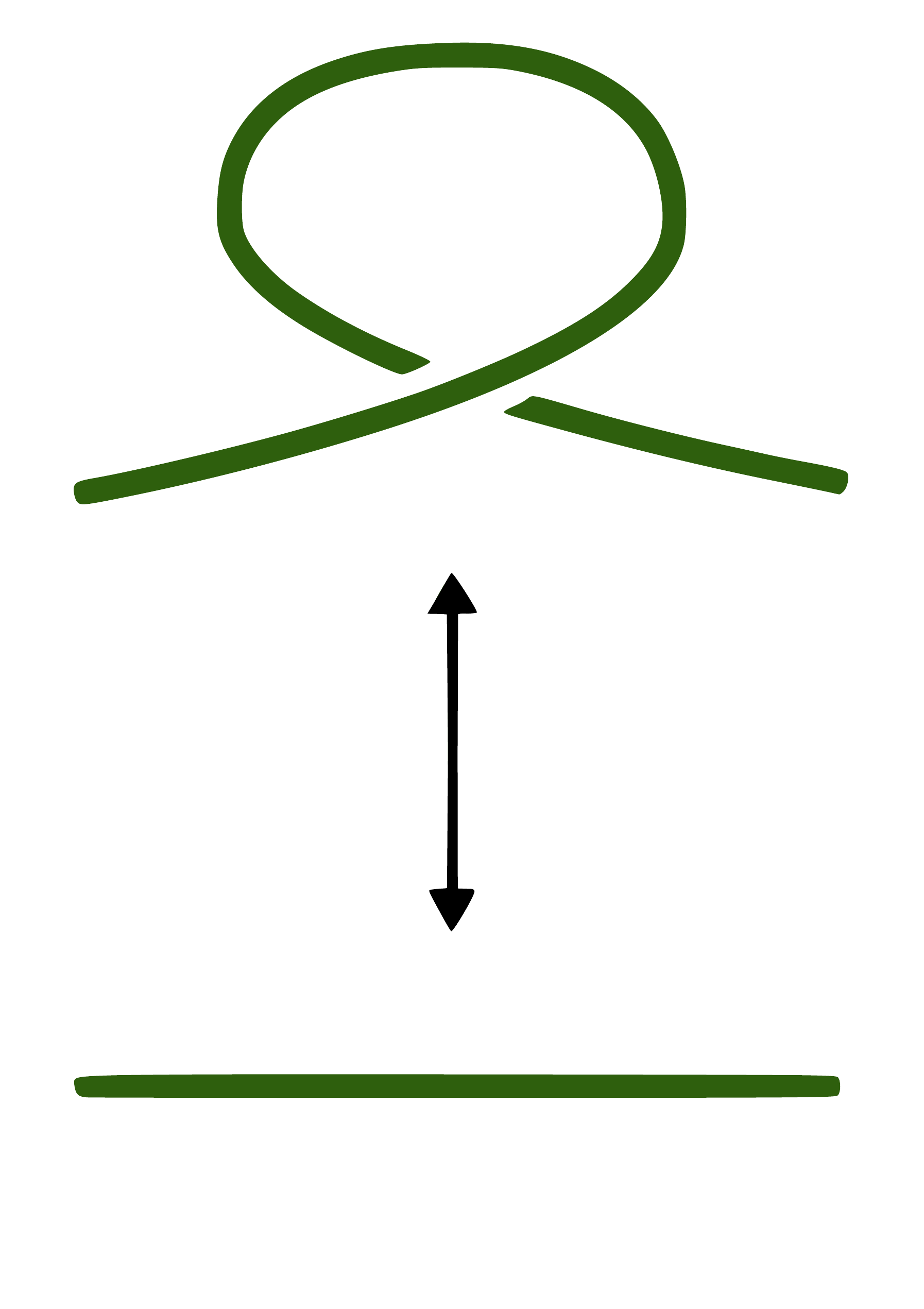}
  \hspace{17ex}
\includegraphics[height=6.2cm]{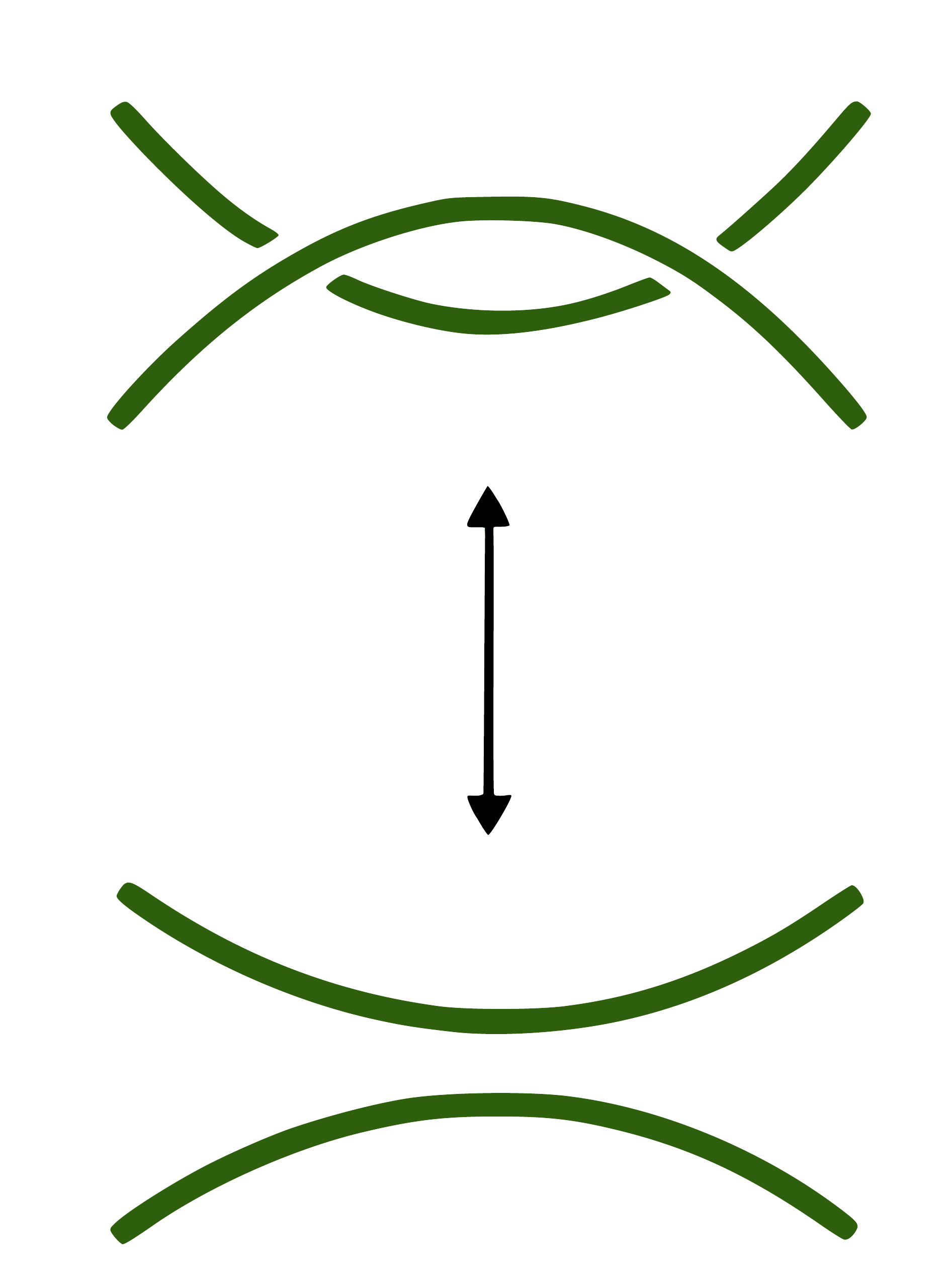}
\vspace{-0.15in}
\end{center}
    \caption{\label{fig:eins} The viewpoint $u$ crosses the tangential surface (left)
    or the edge surface (right). Note that the curve is fixed.
    The arrows indicate the direction of change in viewpoint $u$.
  }
 \end{figure}

\begin{remark}
The locus $\PP^3_\RR \backslash \rho(X)$ of real $X$-border rank $\geq 3$ consists of viewpoints $u$ of crossing-free linear projections of $X_\RR$.
In particular, if $X_\RR$ is a knot  or link, in the usual sense of knot theory,
then every planar projection of $X_\RR$ has a crossing, and hence
  $\rho(X) = \PP^3_\RR$.
\end{remark}

We use classical geometry to
describe the transition between real ranks two and three.
Let $u \in \PP^3_\RR$ and consider the plane curve 
in $\PP^2$ obtained by projecting $X$ from the center $u$.
If $u$ has real $X$-rank two then that plane curve  has a {\em crunode} (ordinary real double point).
As $u$ moves through space and transitions
from real $X$-rank two to real $X$-rank three then  that last crunode disappears. If the transition occurs
via $\tau(X)$ then the intermediate singularity of the projected curve is a cusp.
If it occurs via $\epsilon(X)$ then that singularity is a {\em tacnode}.
The terms ``crunode'' and ``tacnode'' are classical for the relevant real curve singularities.

Figure \ref{fig:eins} shows the transitions as the viewpoint $u$ crosses
the tangential surface $\tau(X)$ and the edge surface $\epsilon(X)$ respectively. The arrows indicate the direction of change in viewpoint of the fixed curve. These are two of the three classical {\em Reidemeister moves}
from knot theory. Transitions via the third Reidemeister move do not cause a change in real $X$-rank.

The edge surface $\epsilon(X)$ plays a prominent role in 
convex algebraic geometry. As shown in \cite{RS2},
it represents the non-linear part in the boundary of the convex hull of $X_\RR$.
See \cite[Figures 1 and 2]{RS2}. In this section  we focus on rational curves.
This allows us to use the methods in \cite[Section 3]{RS2}.
We have the following result about the real rank two boundary.

\begin{proposition}
\label{eq:fourcurves}
There exist rational curves $X_1,X_2,X_3$ and $X_4$ in $\,\PP^3$
such that
$$ 
\partial_{\rm alg}(X_1) = \tau(X_1), \,\,\, 
\partial_{\rm alg}(X_2) = \epsilon(X_2) \cup \tau(X_2),\,\,\, 
\partial_{\rm alg}(X_3) = \emptyset ,\,\,\, {\rm and} \,\,\,
\partial_{\rm alg}(X_4) = \epsilon(X_4).
$$
\end{proposition}

\begin{proof}
By Theorem~\ref{cor:segver},
the twisted cubic curve in Example \ref{ex:binarycubic1} can serve as the curve $X_1$.
The quartic curve in Example \ref{spacecurve1} serves
as  $X_2$. For $X_3$ we  take the Morton curve discussed 
in \cite[Example 4.4]{RS2}. This is rational of degree six
and forms a trefoil knot \cite[Figure 3]{RS2}. 

Rational curves $X_4$ in $\PP^3_\RR$ with  $\partial_{\rm alg}(X_4) = \epsilon(X_4)$
are a bit harder to find. A piecewise-linear
connected example, resembling a 3D Peano curve, can be constructed in two steps. First, we make a curve from six edges of the unit cube. Starting from $(0,0,0)$, the curve travels to $(1,1,1)$ via intermediate vertices $(1,0,0)$ and $(1,1,0)$, and then loops back to $(0,0,0)$ via intermediate vertices $(0,1,1)$ and $(0,0,1)$. In the middle third of each line segment we insert a piecewise linear detour of height $\frac{1}{2}$ in the direction of the next segment.
Four views of this space curve are shown in Figure \ref{x4}.
There are relatively few viewpoints from which the curve has no crossings. From such positions, crossings are always gained in pairs, via transitions along edges,
as shown on the right in Figure \ref{fig:eins}.

The existence of a rational algebraic curve
$X_4$ with the same property can be concluded from the Weierstrass Approximation Theorem. To exclude the possibility that the algebraic boundary is strictly contained in the edge variety, it suffices to show the existence of an approximating curve 
whose edge variety is irreducible.
 This can be ensured using \cite[equation (3.6)]{RS1}, as the rational curve $X_4$
 can be parametrized by  sufficiently generic polynomials.
\end{proof}

\begin{figure}
  \begin{center}
  \includegraphics[width=16cm]{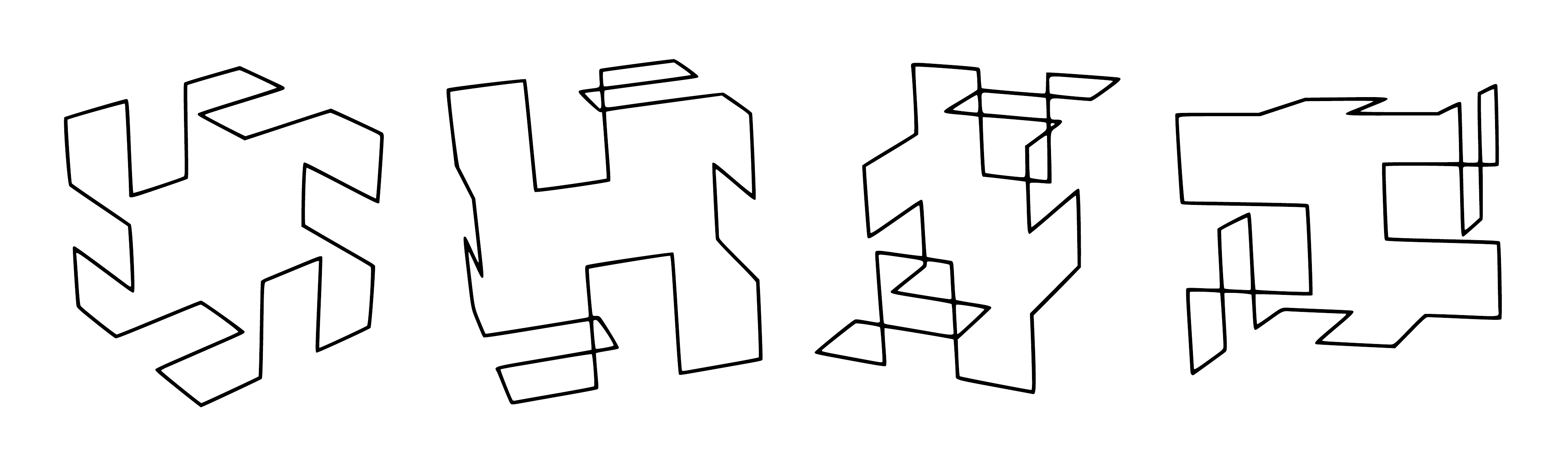}
    \caption{\label{x4}The space curve $X_4$ from Proposition \ref{eq:fourcurves}, as seen from four different angles.}
  \end{center}
  \vspace{-0.15in}
 \end{figure}

In what follows we review the techniques in \cite[pages 7-9]{RS2},
and we show how they can be adapted
for computing rank two decompositions.
Suppose that $X$ is a rational curve of degree $d $ that spans
$\PP^3$.  Note that $\sigma(X) = \PP^3$.
The curve $X$ has a rational parametrization
$$ \PP^1 \rightarrow \PP^3, \,\,
(s:t) \mapsto \bigl(F_0(s,t): F_1(s,t): F_2(s,t): F_3(s,t) \bigr). $$
Here, $F_0,F_1,F_2,F_3$ are binary forms of degree $d$.
Two points $(s_1:t_1)$ and $(s_2:t_2)$ in $\PP^1$ parametrize two distinct points on the curve $X$,
namely
$$ \begin{matrix} & \bigl(F_0(s_1,t_1): F_1(s_1,t_1): F_2(s_1,t_1): F_3(s_1,t_1) \bigr) \\
\quad {\rm and} & \bigl(F_0(s_2,t_2): F_1(s_2,t_2): F_2(s_2,t_2): F_3(s_2,t_2) \bigr).
\end{matrix}
$$
The secant line spanned by these two points in $\PP^3$ 
is characterized by its vector of {\em Pl\"ucker coordinates}
$(p_{01}: p_{02} : p_{03} : p_{12}: p_{13} : p_{23}) \in \PP^5$. These coordinates are
\begin{equation}
\label{eq:plucker}
 p_{ij} \,\, = \,\, \frac{F_i(s_1,t_1) F_j(s_2,t_2) - F_i(s_2,t_2) F_j(s_1,t_1)}{ s_1 t_2-s_2t_1}  \qquad
\,{\rm for} \,\,\, 0 \leq i < j \leq 3  .
\end{equation}
Each Pl\"ucker coordinate $p_{ij}$ is invariant under swapping $(s_1:t_1)$ and $(s_2:t_2)$.
We wish to express $p_{ij}$ as a function of the unordered pair $\{(s_1:t_1),(s_2:t_2)\}$. The two points in $\PP^1$ are represented by the two linear factors of a binary quadric 
\begin{equation}
\label{eq:binaryquadric}
 a x^2 + b xy + c y^2 \,\, = \,\, (s_1x + t_1y)(s_2x+t_2 y). 
 \end{equation}
 We can use equation \eqref{eq:binaryquadric} to write $p_{ij}$
as a homogeneous polynomial of degree $d-1$ in $(a,b,c)$.
The resulting formulas define a rational map 
from $ \PP^2 = {\rm Sym}_2(\PP^1) $
into the Grassmannian of lines $ {\rm Gr}(1,\PP^3)$.
This parametrizes the secant lines.
Standard properties of the Pl\"ucker coordinates imply
that the points $(w:x:y:z) \in \PP^3$ on a particular secant line are the solutions of the linear system of equations
\begin{equation}
\label{eq:4by4} \begin{pmatrix}
 0 &  \phantom{-} p_{23} & -p_{13} & \phantom{-} p_{12}\, \\
-p_{23} &  0 & \phantom{-} p_{03} & -p_{02} \,\\
\phantom{-} p_{13} & -p_{03} & 0 &  \phantom{-} p_{01} \,\\
-p_{12} & \phantom{-}   p_{02} & -p_{01} & 0 \,\\
\end{pmatrix}
\cdot
\begin{pmatrix} w \\ x \\ y \\ z \end{pmatrix} \,\,=\,\,
\begin{pmatrix} 0 \\ 0 \\ 0 \\ 0 \end{pmatrix}.
\end{equation}
As seen from our parametrization in (\ref{eq:plucker}) and
(\ref{eq:binaryquadric}), this is a system of equations of bidegree $(d-1,1)$
in  the pair $\bigl((a,b,c), (w,x,y,z) \bigr)$. They define a threefold in
$\PP^2 \times \PP^3$.  The $X$-rank two decompositions of a data vector $ (w,x,y,z)$  in $\RR^4$ are its fiber under the map that projects the threefold
onto the second factor $\PP^3$.
To be concrete,  given real numbers $w,x,y,z$,
we plug them into  the system (\ref{eq:4by4}).
This yields  four homogeneous equations of degree $d-1$ in
three unknowns $a,b,c$, of which two are linearly independent. The $X$-rank two decomposition defined by a triple $(a,b,c)$ is then obtained from equation \eqref{eq:binaryquadric}. The reality of the $X$-rank two decomposition is determined as follows.

\begin{proposition}
\label{eq:solveforabc}
The point  $u = (w : x : y : z)$ has real $X$-rank $\leq 2$ if and only if
the system (\ref{eq:4by4}) has a real solution $(a:b:c)$
such that the matrix in (\ref{eq:4by4}) is non-zero and
the discriminant $b^2 - 4ac$ of the quadric
(\ref{eq:binaryquadric}) is positive. Such points $(a:b:c) \in \PP^2_\RR$ are in bijection with
lines in $\PP^3_\RR$ that pass through $u$ and meet 
the curve $X$ in two real points.
\end{proposition}

The boundary $\partial(\rho(X))$  marks the transition between
systems (\ref{eq:4by4}) that admit solutions as described
in Proposition~\ref{eq:solveforabc} and those that do not.
The discriminantal surface in $\PP^3$ that separates 
real $X$-rank $\leq 2$ from real $X$-rank $\geq 3$ can 
have components contributed by both the tangential variety
$\tau(X)$ and the edge surface $\epsilon(X)$.
These are ruled surfaces.
The lines in the rulings are represented by curves
in the plane $\PP^2$ with coordinates $(a:b:c)$.
To obtain the surface from each curve, we compute 
its image under the correspondence (\ref{eq:4by4}).

The first relevant curve is the conic $b^2-4ac$. This encodes binary quadrics (\ref{eq:binaryquadric})
with a double root, so its image in $\PP^3$ is the
tangential surface $\tau(X)$. The second relevant curve
has degree $2(d-3)$. Its defining polynomial $\Phi(a,b,c)$ was
constructed in \cite[equation (3.6)]{RS2}. The image of the curve
$\Phi=0$ under the correspondence (\ref{eq:4by4})
is the edge surface $\epsilon(X)$ in $\PP^3$.

\begin{example}\label{spacecurve1} \rm
Let $d=4$ and fix the smooth monomial curve $X$ in $\PP^3$ with parametrization
\begin{equation}
\label{eq:paracurve}
 F_0 = s^4, \,\, F_1 = s^3 t , \,\, F_2 = s t^3, \, \, F_3 = t^4. 
\end{equation}
The parametrization  (\ref{eq:plucker})
of the secant lines of $X$ 
in terms of Pl\"ucker coordinates  is
$$
\begin{matrix}
p_{01}  & = & a^3 ,& \phantom{momo} 
p_{02} & = & a (b^2-ac), & \phantom{momo} 
p_{03} & = & b(b^2-2ac), \\
 p_{12} & =  & abc  ,& \phantom{momo} 
  p_{13} & = & c(b^2-ac), & \phantom{momo} 
  p_{23} & = & c^3 .\\
\end{matrix}
$$
The secant correspondence in $\PP^2 \times \PP^3$ is obtained by
substituting these expressions into (\ref{eq:4by4}), and saturating
with respect to the $p_{ij}$. Its map onto $\PP^3$ has degree three.
A general point $u = (w:x:y:z)$ in $\PP^3$ lies on three
secants, each represented by a point $(a:b:c)$.
The semi-algebraic set $\rho(X)$ consists of points where
at least one of the three secants is real and meets
$X$ in two real points. Algebraically, 
we desire that $a,b,c$ are real and satisfy $b^2 \geq 4ac$.

The tangential surface $\tau(X)$ has degree $6$. We compute its defining equation
as follows. First add $\,b^2-4ac\,$ to the ideal in (\ref{eq:4by4}), then saturate by
the entries of the skew-symmetric $4 \times 4$-matrix, and finally eliminate the unknowns
$a,b,c$. The result is the polynomial
\begin{equation}
\label{eq:tangential4}
16x^3y^3-27w^2y^4+6wx^2y^2z-27x^4z^2+48w^2xyz^2-16w^3z^3.
\end{equation}
The edge surface $\epsilon(X)$ has degree $6$ as well. 
Following \cite[equation (3.6)]{RS2}, it is encoded 
by the plane quadric $\, \Phi(a,b,c) \, = \, b^2 + 2ac$.
The same elimination process yields the polynomial
\begin{equation}
\label{eq:edge4}
32 x^3 y^3-27 w^2 y^4-6 wx^2y^2z - 27x^4z^2+24w^2xyz^2+4w^3z^3.
\end{equation}
The ruled sextic surfaces  (\ref{eq:tangential4}) and (\ref{eq:edge4})
divide $\PP^3_\RR$ into various
connected components. The real rank two locus $\rho(X)$ is the union of
the components whose points obey
Proposition \ref{eq:solveforabc}.

We claim that $\rho(X)$ is a proper subset of $\PP^3_\RR$ and that
both the edge surface  $\epsilon(X)$ and the tangential surface $\tau(X)$
 contribute to the real rank boundary $\partial(\rho(X))$.
 To prove this claim, we consider the line segment in $\PP^3_\RR$ whose points $u(t)$
 are given by the parametrization
$$   w =  84 - 74t,\,\,  x =  13 + 59t,\,\,  y =  62 - 19t, \,\, z = -38 - 10t. $$
Here $t$ is a real parameter that runs from $0$ to $1$.
By substituting into (\ref{eq:tangential4}) and (\ref{eq:edge4}) respectively, we find that the
 line segment crosses the tangential surface $\tau(X)$ twice, namely when
$$ t_1 = 0.41616468475415957221  \quad {\rm and} \quad  t_3 = 0.64786245578375696533. $$
It also crosses the edge surface $\epsilon(X)$ twice, namely at the points
$u(t_2)$ and $u(t_4)$ given by
$$ t_2 = 0.50734775284175190900  \quad {\rm and} \quad  t_4 = 0.81105706603104911043. $$
The above expressions are numerical approximations to the $t_i$. The true values have algebraic degree six, which is the degree of the surfaces $\tau(X)$ and $\epsilon(X)$. 
The $t_i$ cannot be expressed in terms of radicals over $\mathbb{Q}$  because
the Galois group is the symmetric group on six letters.

The value of the parameter $t$ divides the line segment into five smaller segments on which
the corresponding secant lines of X have constant real behavior. Computations reveal:
\begin{itemize}
\item For $0 < t < t_1$, the real $X$-rank of $u(t)$ is $3$. One of the three complex 
 secant lines is real but it does not meet the curve $X$ in real points.
\item For $ t_1 < t < t_2 $, the real $X$-rank of $u(t)$ is $2$. One of the three complex 
 secant lines is real and it meets the curve $X$ in two real points.
\item For $t_2 < t < t_3$, the real $X$-rank of $u(t)$ is $2$. All the three complex 
 secant lines are real and they all meet the curve $X$ in two real points.
\item For $t_3 < t < t_4$, the real $X$-rank of $u(t)$ is $2$. All the three complex 
 secant lines are real but only two of them meet the curve $X$ in two real points.
\item For $t_4 < t < 1$ the real $X$-rank of $u(t)$ is $3$. One of the three complex 
 secant lines is real but it does not meet the curve X in real points.
 \end{itemize}
 This verifies that  both of the transitions depicted in Figure \ref{fig:eins}
do occur along this line segment.
  At $t = t_1$ the real $X$-rank changes by crossing the tangential surface,
   and at $t = t_4$ it changes by crossing the edge surface.
Additional crossings of the two boundary surfaces take place at $t=t_3$ and
at $t=t_2$, but these do not
change the real $X$-rank of $u(t)$.
   
   We finally note that both of the two scenarios (d) and (e) 
   for rank two approximation,   discussed prior to
   Theorem \ref{thm:notrank1}, are realized for $X$ along this line segment.
   Namely, for sufficiently small $\epsilon > 0$, 
   we obtain (d) for    $u=u(t_1 - \epsilon)$,
   and we obtain (e) for $ u=u(t_4  + \epsilon)$.
\end{example}
 
\section{Tensors and their Hyperdeterminants} \label{sec:tensors}

The varieties $X$  whose ranks are
most relevant for applications are the Segre variety and the Veronese variety. When studying tensors of format $n_1 \times n_2 \times \cdots \times  n_d$,
we set $N = n_1 n_2 \cdots n_d - 1$ and 
$X \subset \PP^N$ is the {\em Segre variety} whose points are tensors of rank one.
When  studying symmetric tensors of format
$n \times n \times \cdots \times n$ with $d$ factors, we set
$N = \binom{n+d-1}{d} - 1$ and $X \subset \PP^N$ is the
{\em Veronese variety} whose points are symmetric tensors of rank one.
These two classical varieties $X$ are non-defective provided $d \geq 3$. 
We exclude the case $d=2$ because the corresponding varieties of rank one matrices are defective.

For any variety $X$ as before,
the degree of the natural parametrization
 of its secant variety $\sigma(X)$ gives the number of rank two decompositions of a generic 
 point. It is the integer $k$ in the proof of Theorem \ref{thm:main}. If the parametrization
 is birational ($k=1$) then $\sigma(X)$ is said to be
 {\em identifiable}. If the secant variety $\sigma(X)$ is identifiable
 then there is no edge variety $\epsilon(X)$.
  
  \begin{remark} \label{taugeneral} \rm
  It is natural to wonder whether    $\tau(X)_\RR$ is
  always contained in the real rank two locus  $\rho(X)$.
    Furthermore, if $\sigma(X)$ is identifiable, 
    then    $\tau(X)_\RR \subseteq \partial(\rho(X))$ seems plausible.
    This would be true if  every transition through $\tau(X)_\RR$ were as in Case 1 of
    Theorem~\ref{thm:main}.  However, this may be false.
    For instance, consider a smooth space curve $X$ as in Section~\ref{curves}.
     The tangential surface  $\tau(X)_\RR$ can look locally like a {\em Whitney umbrella}.
     It might have lower-dimensional
    real pieces that protrude into the interior of $\rho(X)$ or its complement.
If $\sigma(X)$ is not identifiable then the interior of $\rho(X)$ can  contain a
region of $\tau(X)_\RR$ that is Zariski dense in $\tau(X)$. The point $u(t_3)$ in  Example~\ref{spacecurve1}
lies inside such a region of the tangential surface.
    \end{remark}
  
We now focus on the case of tensors, where $X$ is a Segre of Veronese variety. Here, the secant variety is usually identifiable, and Remark \ref{taugeneral} can be strengthened as follows.

\begin{lemma}\label{tautensor}
Let $X$ be a Segre variety or Veronese variety with $d \geq 3$. Then the
real tangential variety is contained in the real rank two locus;
in symbols, $\tau(X)_\RR \subseteq \partial(\rho(X))$.
\end{lemma}

\begin{proof}
Let $T$ be a real point in $\tau(X)$. It is expressible as a sum of $d$ rank one tensors,
$$ T = y_1 \otimes x_2 \otimes \cdots \otimes x_d + x_1 \otimes y_2 \otimes x_3 \otimes \cdots \otimes x_d + \cdots + x_1 \otimes \cdots \otimes x_{d-2} \otimes y_{d-1} \otimes x_d + x_1 \otimes \cdots \otimes x_{d-1} \otimes y_d ,$$
where we omit the subscripts for the Veronese case.
This representation is derived in \cite{Lan}.
Direct computations shows that there exists a sequence of tensors $a_n \to T$ where each $a_n$ lies on the secant line spanned by $(x_1 + \frac{1}{d} y_1) \otimes (x_2 + \frac{1}{d} y_2) \otimes \cdots \otimes (x_d + \frac{1}{d}y_d)$ and $x_1 \otimes x_2 \otimes \cdots \otimes x_d$. There also exists a sequence of tensors $b_n \to T$ where each $b_n$ lies on the secant line spanned by $(x_1 + \frac{1}{di} y_1) \otimes (x_2 + \frac{1}{di} y_2) \otimes \cdots \otimes (x_d + \frac{1}{di} y_d)$ and $(x_1 - \frac{1}{di} y_1) \otimes (x_2 - \frac{1}{di} y_2) \otimes \cdots \otimes (x_d - \frac{1}{di} y_d)$. 
Here $i  =\sqrt{-1}$.
See Example~\ref{2x2x2x2} for the case when $X$ is a rational normal curve.
By Kruskal's Theorem \cite[\S 3.2]{KB}, these real (resp. complex) expressions for $a_n$ (resp. $b_n)$ are unique, whenever three or more $y_i$ are non-zero. Therefore, $T$ is both a limit of real rank two tensors, and a limit of tensors that are not in $\rho(X)$, hence it lies in the boundary $ \partial (\rho(X))$.

It remains to consider the case when at most two $y_i$ are non-zero. Then $T = M \otimes x$ with $M$ a matrix and $x$ a rank one tensor. One can construct sequences of rank one tensors, $\alpha_n , \beta_n \to x$, with $\alpha_n$ real and $\beta_n$ complex by perturbing $x$ by arbitrarily small real (resp. complex) rank one tensors. Then $a_n = M \otimes \alpha_n \to T$ and $b_n = M \otimes \beta_n \to T$ are real and complex sequences respectively, and we conclude as above.\end{proof}

We have the following characterization of the algebraic real rank two boundary for tensors.

  \begin{theorem} \label{cor:segver}
 Let $X$ be the Segre variety (resp.~the Veronese variety) whose points are
 $d$-dimensional   tensors (resp.~symmetric tensors)  of rank one.  If $d \geq 3$ then 
 the algebraic real rank two boundary of $X$ is non-empty and
    equals the tangential variety of $X$. In symbols,
 $$ \partial_{\rm alg}(\rho(X)) \,\, = \,\, \tau(X). $$
 \end{theorem}
 
 \begin{proof}
The secant variety $\sigma(X)$ is identifiable, since Kruskal's Theorem holds 
generically for rank two tensors. Therefore $\epsilon(X)$ does not exist, since points on $\epsilon(X)$ are limits of tensors lying on at least two distinct secant lines; see Remark \ref{rem:edge}. 
To prove the theorem, we must exclude the possibility $\partial_{\rm alg}(\rho(X)) = \emptyset$. 
By taking sums of complex conjugate pairs of points on the affine cone $\hat X$,
 one creates many tensors that lie in $\sigma(X)_\RR$ but not in $\rho(X)$.
Hence the rank two locus $\rho(X)$ has a non-empty boundary
inside $\sigma(X)_\RR$, and the algebraic boundary
$ \partial_{\rm alg}(\rho(X))$ is a non-empty hypersurface in $\sigma(X)$. 
That hypersurface is contained in the irreducible hypersurface $\tau(X)$, 
by Theorem~\ref{thm:main}. This implies that they are equal.
  \end{proof}

We next derive the following general result concerning tensors $T$ of
arbitrary format $n_1 \times n_2 \times \cdots \times n_d$ where $d \geq 3$.
A $2 \times 2 \times 2$ sub-tensor of $T$ has coordinates in which $d-3$ of the indices are fixed, and the remaining three can take one of two different values.
We are interested in the  hyperdeterminants of these sub-tensors.
These are the $2 \times 2 \times 2$ {\em sub-hyperdeterminants} of $T$.
Their number is found to be
\begin{equation}
\label{eq:theirnumber}
\frac{1}{8} \cdot n_1 n_2  n_3 \cdots n_d \,\cdot \!\ \!\!\!\!\! \sum_{1 \leq i < j < k \leq d} (n_i-1)(n_j-1)(n_k-1).
\end{equation}

\begin{theorem} \label{thm:subhyper}
A real tensor $T$ has real border rank $\leq 2$ if and only if all of its flattenings have rank $\leq 2$ 
and all of its  $2 \times 2 \times 2$ sub-hyperdeterminants are non-negative.  If this holds then 
the real rank of $T$ is exactly two if at least one of the flattenings  of $T$ has rank two  
and at least one of the $2 \times 2 \times 2$ sub-hyperdeterminants of $T$ is strictly positive.
\end{theorem}

\begin{proof}
We begin with the only-if direction of the first statement.
Let $T$ have real border rank $\leq 2$. Then every $2 \times 2 \times 2$ sub-tensor
$T'$  has real border rank $\leq 2$. We can approximate $T'$
by a sequence of tensors $T''$ that have real rank two.
The entries $t''_{ijk}$ of any tensor in the approximating sequence can be written as
$\,t''_{ijk} = a_i b_j c_k + d_i e_j f_k,$
where the parameters are real. With a computation 
one checks that the hyperdeterminant of $T''$ evaluates to
$$\,(a_1 d_2 - a_2 d_1)^2 (b_1 e_2 - b_2 e_1)^2 (c_1 f_2 - c_2 f_1)^2 .$$
 This quantity is non-negative since all parameters are real.
By continuity, we conclude that all $2 \times 2 \times 2$ sub-hyperdeterminants
of the original tensor $T$ are non-negative.

For the if direction, suppose that $T$ is a tensor in $\sigma(X)_\RR$
whose $2 \times 2 \times 2$ sub-hyperdeterminants are all non-negative.
The complex rank of $T$ is either $1$, $2$ or $\geq 3$. If it is $1$
then $T$ is in the real Segre variety $X_\RR$ and hence in $\rho(X)$.
If $T$ has complex rank $\geq 3$ then it is in $\tau(X)_\RR \backslash X$, and we deduce that $T \in \rho(X)$ from Lemma \ref{tautensor}.

It remains to examine the case when $T$ has complex rank two
and real rank $\geq 3$. The tensor $T$ lies on a real secant line,
spanned by a pair of complex conjugate points in $X$. Consider any $2 \times 2 \times 2$ sub-tensor $T'$ of $T$. We can write the entries
$t'_{ijk}$ of $T'$ as
$$ t'_{ijk} \, = \,
(a_i + A_i \sqrt{-1})(b_j + B_j \sqrt{-1})(c_k + C_k \sqrt{-1}) +
(a_i - A_i \sqrt{-1})(b_j - B_j \sqrt{-1})(c_k - C_k \sqrt{-1}) , $$
where the parameters $a,b,c,A,B,C$ are real.
One checks that the hyperdeterminant of $T'$ is
\begin{equation}
\label{eq:sos}
 -(a_1 A_2 - a_2 A_1)^2  \cdot
(b_1 B_2 - b_2 B_1)^2  \cdot
(c_1 C_2 - c_2 C_1)^2 \cdot 4^3 .
\end{equation} 
This expression is non-positive since all parameters are real.
Our hypothesis that all $2 \times 2 \times 2$ sub-hyperdeterminants are non-negative means they must all be zero.

The rank two representation of $T$ involves pairs  of  vectors $\{a,A\} \subset \RR^{n_1}$,
$ \{b,B\} \subset \RR^{n_2}$, $ \{c,C\} \subset \RR^{n_3} , \ldots$
Every $2 \times 2 \times 2$
sub-hyperdeterminant of $T$ has the form in \eqref{eq:sos} and equates~to zero. From this we conclude that,
for all but two of the  pairs 
$\{a,A\}, \{b,B\}, \{c,C\} , \ldots$, the vectors in the pair are linearly dependent. If not, we could choose indices $(i,j)$ from each vector pair for which the expression $a_i A_j - a_j A_i$ does not vanish, yielding a non-vanishing sub-hyperdeterminant. Hence $T$ is the tensor product of a matrix with $d-2$ vectors. 
This contradicts the hypothesis that $T$ has real rank exceeding~two.

If $T$ is rank one, all flattenings have rank one and all $2 \times 2 \times 2$ sub-hyperdeterminants vanish. So if one flattening has rank two, or one sub-hyperdeterminant is strictly positive, the real rank of $T$ must be at least two. To conclude the proof, it remains to consider tensors in $\rho(X)$ whose real rank exceeds two but are nonetheless there due to taking the closure.

Such tensors lie in $\partial (\rho(X))$, and hence in the tangential variety $\tau(X)$.
We claim that all  sub-hyperdeterminants vanish on $\tau(X)$.
  This is immediate in the base case $X = {\rm Seg}( \PP^1 \times \PP^1 \times \PP^1) $ in $\PP^7$, since 
  $\tau(X)$ equals the vanishing locus of the hyperdeterminant. For larger tensor formats, 
  the projection of the tangential variety 
to any $2 \times 2 \times 2$ sub-tensor is precisely that same tangential
variety. Hence each $2 \times 2 \times 2$ sub-hyperderminant
vanishes on $\tau(X)$, for  Segre varieties $X$ of arbitrary size. Thus, 
if a tensor has at least one flattening of rank two, and at least one 
sub-hyperdeterminant strictly positive, it has real rank exactly two.
\end{proof}

\begin{example} \rm
It is instructive to work through this proof for $2 \times 2 \times 2 \times 2$-tensors $T$.
If $T$ has complex rank two and real rank $\geq 3$ then its entries $t_{ijkl}$ have the
parametric representation
$$  \begin{matrix}t_{ijkl} \,\,\,= \,\, &&
 (a_i + A_i \sqrt{-1})(b_j + B_j \sqrt{-1})(c_k + C_k \sqrt{-1})(d_l+ D_l \sqrt{-1}) \\
&+ & (a_i - A_i \sqrt{-1})(b_j - B_j \sqrt{-1})(c_k - C_k \sqrt{-1}) (d_l - D_l \sqrt{-1}) .
\end{matrix} $$
Suppose  the eight $2 \times 2 \times 2$ sub-hyperdeterminants of $T$ are all non-negative. They are
\begin{equation}
\label{eq:exp222}
\begin{matrix}
 - (a_0^2 + A_0^2)^2 (b_0 B_1 - b_1 B_0)^2 (c_0 C_1 - c_1 C_0)^2 (d_0 D_1 - d_1 D_0)^2 4^3, \\
 - (a_1^2 + A_1^2)^2 (b_0 B_1 - b_1 B_0)^2 (c_0 C_1 - c_1 C_0)^2 (d_0 D_1 - d_1 D_0)^2 4^3, \\
 - (b_0^2 + B_0^2)^2 (a_0 A_1 - a_1 A_0)^2 (c_0 C_1 - c_1 C_0)^2 (d_0 D_1 - d_1 D_0)^2 4^3, \\
 - (b_1^2 + B_1^2)^2 (a_0 A_1 - a_1 A_0)^2 (c_0 C_1 - c_1 C_0)^2 (d_0 D_1 - d_1 D_0)^2 4^3, \\
 - (c_0^2 + C_0^2)^2 (a_0 A_1 - a_1 A_0)^2 (b_0 B_1 - b_1 B_0)^2 (d_0 D_1 - d_1 D_0)^2 4^3, \\
 - (c_1^2 + C_1^2)^2 (a_0 A_1 - a_1 A_0)^2 (b_0 B_1 - b_1 B_0)^2 (d_0 D_1 - d_1 D_0)^2 4^3, \\
 - (d_0^2 + D_0^2)^2 (a_0 A_1 - a_1 A_0)^2 (b_0 B_1 - b_1 B_0)^2 (c_0 C_1 - c_1 C_0)^2 4^3, \\
 - (d_1^2 + D_1^2)^2 (a_0 A_1 - a_1 A_0)^2 (b_0 B_1 - b_1 B_0)^2 (c_0 C_1 - c_1 C_0)^2 4^3.
 \end{matrix}
 \end{equation}
 We note that the first factor does not appear in equation \eqref{eq:sos} because the fixed indices were subsumed into the expressions for one of the parameter pairs $\{a,A\}, \{b,B\}, \{c,C\}$.
 
 It cannot be that  $a_0,A_0,a_1,A_1$ are all zero, and similarly for the
 other letters. Hence
 $$ \begin{matrix}
 (a_0 A_1 - a_1 A_0) (b_0 B_1 - b_1 B_0) (c_0 C_1 - c_1 C_0)= 
  (a_0 A_1 - a_1 A_0) (b_0 B_1 - b_1 B_0) (d_0 D_1 - d_1 D_0) = \phantom{0} \\
  (a_0 A_1 - a_1 A_0) (c_0 C_1 - c_1 C_0) (d_0 D_1 - d_1 D_0) = 
 (b_0 B_1 - b_1 B_0) (c_0 C_1 - c_1 C_0) (d_0 D_1 - d_1 D_0) = 0. 
 \end{matrix} $$
Two of the four factors are zero. There are six cases.
Up to relabeling,  
$\,a_0 A_1 - a_1 A_0 = b_0 B_1 - b_1 B_0 = 0$.
This implies that $T = (a_0,a_1)\otimes  (b_0,b_1) \otimes U$,
where $U$ is a $2 \times 2$-matrix. Clearly $U$ has real rank $\leq 2$.
This shows that $T$ has real rank $\leq 2$, the necessary contradiction. \end{example}

We briefly discuss the implications of our inquality 
description of the real rank two locus $\rho(X)$ for its boundary
$\partial(\rho(X))$. Here $X$ is the
Segre variety of rank one tensors. 
Let ${\rm Hyp}$ denote the variety consisting of 
tensors whose $2 \times 2 \times 2$ sub-hyperdeterminants are all zero.

\begin{proposition} \label{prop:strict}
The real rank two boundary $\partial(\rho(X))$ is a subset of the
semi-algebraic set ${\rm Hyp} \cap \rho(X)$.
This containment is an equality for tensors of order $d=3$ but strict for $d \geq 4$. That is, when $d \geq 4$ there exist tensors on ${\rm Hyp}$ which lie in the interior of $\rho(X)$.
\end{proposition}

\begin{proof}
We saw in Theorem \ref{cor:segver}
that $\partial (\rho(X))$ is contained in the tangential variety
$\tau(X)$. So for the first claim,
it suffices to show that each $2 \times 2 \times 2$ sub-hyperdeterminant
vanishes on $\tau(X)$. This was shown in end of the proof of Theorem \ref{thm:subhyper};
see also~\cite[Theorem 1.3]{Oed}.

Suppose that $d = 3$ and $T = (t_{ijk})$ is an 
$n_1 \times n_2 \times n_3$-tensor in $  {\rm Hyp} \cap \rho(X)$.
If $T$ lies in $\tau(X)$ then it is in the boundary $\partial(\rho(X))$, by Lemma \ref{tautensor}. We may therefore assume
that $T$ has real rank $\leq 2$. So, its entries can be written as
$\,t_{ijk} = a_i b_j c_k + d_i e_j f_k$.
Since $T \in {\rm Hyp}$, for all indices
$\, 1 \leq i_1 < i_2 \leq n_1$,
$\, 1 \leq j_1 < j_2 \leq n_2\,$ and
$ \,1 \leq k_1 < k_2 \leq n_3$, we have
$$ 
(a_{i_1} d_{i_2} - a_{i_2} d_{i_1}) \cdot
(b_{j_1} e_{j_2} - b_{j_2} e_{j_1})  \cdot
(c_{k_1} f_{k_2} - c_{k_2} f_{k_1}) \quad = \quad 0  .
$$
This condition implies that
either $\{a,d\}$ or $\{b,e\}$ or $\{c,f\}$
are linearly dependent. After relabeling and rescaling
we may assume $a = d$.
This implies $T = a \otimes \bigl( (b \otimes c) + (e \otimes f)  \bigr)$.
This tensor lies in the tangential variety $\tau(X)$ of
the Segre variety $X = {\rm Seg}(\PP^{n_1-1} \times
\PP^{n_2-1} \times \PP^{n_3-1})$.

It remains to show that $\partial(\rho(X))$ is strictly
contained in ${\rm Hyp} \cap \rho(X)$ for $d \geq 4$.
Consider $d=4$ and $X = {\rm Seg}({(\PP^1)}^4)$.
Let $\{e_1,e_2\}$ be the standard basis of $\RR^2$.
The rank two tensor
\begin{equation} 
\label{eq:x^4+y^4}
 T \,\,\,=\,\,\,
e_1 \otimes  e_1 \otimes  e_1 \otimes  e_1 \,+\,
e_2 \otimes  e_2 \otimes  e_2 \otimes  e_2
\end{equation}
is in the relative interior of the real rank two locus $\rho(X)$.
All eight $2 \times 2 \times 2$ sub-tensors have rank one,
so the eight hyperdeterminants vanish. Hence
$T$ lies in ${\rm Hyp} \cap \rho(X) \backslash \partial(\rho(X))$.
This tensor can now be embedded into all larger formats,
and we get the conclusion  for $d \geq 4$.
\end{proof}

\begin{remark} \rm
The tensor (\ref{eq:x^4+y^4}) lies on the interior of $\rho(X)$. All of its $2 \times 2 \times 2$ hyperdeterminants vanish, so it lies on the variety ${\rm Hyp}$. This demonstrates that the only-if direction in the
second sentence of Theorem \ref{thm:subhyper} does not hold.\end{remark}

\begin{remark} \rm
The number (\ref{eq:theirnumber})
of $2 \times 2 \times 2$ sub-hyperdeterminants 
for a tensor of format $n \times n \times \cdots \times n $ equals
$\frac{1}{8}\binom{d}{3} n^d (n-1)^3$.
If the tensor is symmetric then this
number reduces~to
\begin{equation}
\label{eq:theirnumbersym}
\binom{ n+d-4}{ n-1} \binom{ \binom{n}{2}+2}{3}.
\end{equation}
Among these we only need  hyperdeterminants 
  whose expansion as in (\ref{eq:exp222}) is a
 sixth power like $(a_0 A_1 - a_1 A_0)^6$
times an extraneous factor $\prod_i (a_i^2 + A_i^2)^2$.
  That reduces the number to
\begin{equation}
\label{eq:theirnumbersym2}
\binom{ n+d-4}{ n-1} \binom{n}{2}.
\end{equation}
Each of these symmetric hyperdeterminants looks like
the quartic $D$ in the next example.
\end{remark}

\begin{example} 
\label{ex:binarycubic1} \rm
Let $n=2, d=3$. Here $X$ is the twisted cubic curve in $\PP^3$.
The tangential variety $\tau(X)$ is the quartic surface in 
$\sigma(X) = \PP^3$ given by  the discriminant
of a binary cubic:
\begin{equation}
\label{eq:Q_0}
 D \,\,\,= \,\,\, x_0^2 x_3^2 - 6 x_0 x_1 x_2 x_3 - 3 x_1^2 x_2^2+4x_1^3 x_3+4x_0 x_2^3 
\,\,\, = \,\,\, {\rm det}
\begin{small}
\begin{pmatrix}
\, x_0 & 2 x_1 &  x_2 &   0 \, \\
\,  0  &  x_0 & 2 x_1 &  x_2  \, \\
\, x_1 & 2 x_2 &  x_3 &  0 \, \\
\,  0 &  x_1 & 2 x_2 & x_3 \,\,
\end{pmatrix}  
\end{small}
 \end{equation}
This is the $2 {\times} 2 {\times} 2$ hyperdeterminant (\ref{eq:hyperdet}) 
 specialized to symmetric tensors \cite[page 2]{Oed2}.
 Both numbers (\ref{eq:theirnumbersym}) and   (\ref{eq:theirnumbersym2})  are one.
 The real rank two locus $\rho(X)$  is the subset of $\PP^3_\RR$ defined by  $D \geq 0$.
 For a study of hyperdeterminants of symmetric tensors  we refer to    Oeding~\cite{Oed2}.
 \end{example}

 \section{The Tangential Variety of the Veronese} \label{sec:five}
 
The variety $\sigma(X)$ of rank two tensors is defined by
the $3 \times 3$-minors of all flattenings \cite{Rai}.
Among the real points on that secant variety, the locus $\rho(X)$ 
is defined by the hyperdeterminantal inequalities in
Theorem~\ref{thm:subhyper}. Since the algebraic boundary of
$\rho(X)$ is the tangential variety $\tau(X)$, one might think that $\tau(X)$ is obtained by
setting the hyperdeterminants to zero. But this is false, as seen in Proposition \ref{prop:strict}.
Oeding and Raicu \cite{OR, Rai} showed that
 $\tau(X)$ is often defined by quadrics. In this section we
 focus on Veronese varieties, and we 
  translate the representation-theoretic results from \cite{OR} into
 explicit quadrics. We close with examples
 that illustrate the findings in our paper
 for the rational normal curve $X = \nu_d(\PP^1)$.
 
The following result is for  tensors with $d \geq 3$.
The variety $X$ comprises rank one tensors, so it is the Segre variety
$X = {\rm Seg}(\PP^{n_1 - 1} \times \cdots \times \PP^{n_d - 1})$ or 
the Veronese variety $X = \nu_d (\PP^{n-1})$.

\begin{theorem}[Oeding-Raicu \cite{OR}, Raicu \cite{Rai}] \label{thm:oedingraicu}
The ideal of the secant variety $\sigma(X)$ is generated by the
$3 \times 3$-minors of the various flattenings of the tensor. For symmetric tensors, it suffices to
take the $3 \times 3$-minors of the most symmetric catalecticant matrix.
The ideal of the tangential variety $\tau(X)$ is generated in degree at most four;
the Schur modules of  minimal generators are known explicitly. If $d \geq 5$ then
quadrics suffice to generate the ideal of $\tau(X)$.
\end{theorem}

The space of minimal generators of the prime ideals in question is
a $G$-module, where $G = {\rm SL}(n)$ if  $X = \nu_d (\PP^{n-1})$ 
 and $G = {\rm SL}(n_1) \times \cdots \times {\rm SL}(n_d)$ if
$X = {\rm Seg}(\PP^{n_1 - 1} \times \cdots \times \PP^{n_d - 1})$.
The term {\em Schur module} refers to the irreducible representations
that occur in these $G$-modules. We shall use basics from the
 representation theory of $G$, as in  Landsberg's book~\cite{Lan}.

Our aim is to extract explicit polynomials from the last two sentences in Theorem~\ref{thm:oedingraicu},
for the case when $X = \nu_d (\PP^{n-1})$ and
$G = {\rm SL}(n)$. The irreducible $G$-modules of degree $d$ are indexed by
partitions $\lambda$ of  $d$ with at most $n$ parts.  The module
for $\lambda$ is denoted $S_\lambda$. It has a natural basis, labeled by
semi-standard Young tableaux of shape $\lambda$ filled with $\{1,2,\ldots,n\}$.

We shall present a basis for the space $I_2(\tau(X))$ of quadrics  that vanish on $\tau(X)$. 
Clearly, all such quadrics are minimal ideal generators, since $\tau(X)$ does not lie
in a linear subspace of $\PP^{\binom{n+d-1}{d}-1}$.
Proposition \ref{prop:landsberg} says
that $I_2(\tau(X))$ usually defines $\tau(X)$ as a subvariety of $\sigma(X)$.

Fix an even positive integer $k $ and consider the
irreducible $G$-module $S_\lambda(\CC^n)$
where $\lambda$ is the partition  $(2d-k,k)$. We draw $\lambda$
as a shape with two rows, the first of length $2d-k$
and the second of length $k$.  A basis of $S_\lambda(\CC^n)$
is indexed by the semi-standard Young tableaux (SSYT) of shape $\lambda$
filled with integers between $1$ and $n$. The
SSYT of shape $\lambda$ are identified with
pairs $(\mu,\nu)$ of row vectors $\mu \in \{1,2,\ldots,n\}^{2d-k}$ and $\nu \in \{1,2,\ldots,n\}^{k}$
that satisfy
\begin{equation}
\label{eq:munu}  \begin{matrix}
\mu_1 \leq  \mu_2 \leq \mu_3 \leq \cdots \leq \mu_{k} \leq \mu_{k+1} \leq \cdots \leq \mu_{2d-k},\,
\qquad \qquad \, \\
\quad \nu_1 \leq \nu_2 \leq \nu_3 \leq \cdots \leq \nu_{k} \,\,\,\,{\rm and} \,\,\,\,
\mu_i <  \nu_i \,\,\hbox{for} \,\, i = 1,2,\ldots,k.
\end{matrix}
\end{equation}
By the {\em Hook Length Formula}, the number of such SSYT of shape $\lambda$ equals
 \begin{equation}
 \label{eq:HLF}
 {\rm dim}\bigl(S_\lambda(\CC^n)\bigr) \,\, = \,\,\,
 \prod_{i=1}^k \frac{n-1+i}{2d+2-k-i} 
 \cdot \prod_{i=k+1}^{2d-k} \frac{n-1+i}{2d+1-k-i}
  \cdot \prod_{j=1}^k \frac{n-2+j}{k+1-j} .
\end{equation}
  We realize $S_\lambda(\CC^n)$ as a submodule of $(\CC^n)^{\otimes 2d}$
  by assigning to $(\mu,\nu)$ with (\ref{eq:munu}) the basis vector
  \begin{equation}
  \label{eq:wedges} (e_{\mu_1} \wedge e_{\nu_1}) \otimes 
(e_{\mu_2} \wedge e_{\nu_2}) \otimes  \cdots
\otimes     (e_{\mu_k} \wedge e_{\nu_k}) \otimes 
     [\,e_{\mu_{k+1}} e_{\mu_{k+2}} \cdots \,e_{\mu_{2d-k}}].
     \end{equation}
Here $e_1,\ldots,e_n$ is the standard basis of $\CC^n$, the symbol
 $\wedge$ denotes antisymmetrization of the tensor product,
and the expression $[ \,\cdots\, ]$ is the symmetrization of that tensor product.

We next translate the expression (\ref{eq:wedges}) into
a quadratic polynomial in the $\binom{n+d-1}{d}$ 
homogeneous coordinates $x_u$ on 
$\PP^{\binom{n+d-1}{d}-1}$. This polynomial is supposed to 
vanish on $\tau(X)$. We write $(t_1:t_2:\cdots:t_n)$ for
the homogeneous coordinates on $\PP^{n-1}$.
The parametrization of $\sigma(X)$ by pairs of points in the cone over
the Veronese variety $X$ 
can be written as follows:
\begin{equation}
\label{eq:taupara}
\sum_{|u| = d} \binom{|u|}{u} x_u t^u \,\, = \,\,
(a_1 t_1 + a_2 t_2 + \cdots + a_n t_n)^d \,+\,
(b_1 t_1 + b_2 t_2 + \cdots + b_n t_n)^d .
\end{equation}
We translate the expression (\ref{eq:wedges}) into 
the following polynomial in the $2n$ parameters:
\begin{equation}
\label{eq:wedge2}
\prod_{i=1}^k (a_{\mu_i} b_{\nu_i} - a_{\nu_i} b_{\mu_i}) \cdot
\biggl(\sum
a_{\mu_{j_1}} a_{\mu_{j_2}} \cdots a_{\mu_{j_{d-k}}}
b_{\mu_{j_{d-k+1}}} b_{\mu_{j_{d-k+2}}} \cdots b_{\mu_{j_{2d-2k}}}\biggr),
\end{equation}
where the sum is over permutations
$(j_1,j_2,\ldots,j_{2d-2k}) $ of $\{k+1,k+2,\ldots,2d-k\}$ such that 
$$ j_1 < j_2 < \cdots < j_{d-k} \quad {\rm and} \quad
  j_{d-k+1} < j_{d-k+2} < \cdots  < j_{2d-2k} . $$
 The sum in (\ref{eq:wedge2}) has  $\binom{2d-2k}{d-k}$ terms.
 The group $G = {\rm SL}(n)$ acts on the vectors $a$ and $b$, and hence on the
 span of the polynomials  (\ref{eq:wedge2}).  This
   is the irreducible representation $S_\lambda(\CC^n)$.

 \begin{proposition} \label{prop:uniquepreimage}
 The polynomial (\ref{eq:wedge2}) is in the coordinate ring of $\sigma(X)$,
 i.e.~it lies in the image of the  ring homomorphism $\CC[x] \rightarrow \CC[a,b]$ 
 that is given by the parameterization (\ref{eq:taupara}).
 Its preimage in $\CC[x]$ is unique.  That polynomial vanishes on $\tau(X)$ if and only if $k \geq 4$.
 \end{proposition}
 
 \begin{proof}
 Since the index $k$ introduced prior to (\ref{eq:munu}) 
 is even, the polynomial (\ref{eq:wedge2}) is unchanged if we
 switch the two letters $a$ and $b$.
The polynomial (\ref{eq:wedge2}) is invariant in that sense.
  The coefficients of the right hand side of
(\ref{eq:taupara}) span the space of all such invariant polynomials of degree $d$.
This follows from the fact that the usual ring of symmetric polynomials
is generated by the power sums. Hence (\ref{eq:wedge2}) is in the
image of the ring map $\CC[x] \rightarrow \CC[a,b]$.
The kernel of that map is the ideal of the secant variety $\sigma(X)$.
That ideal contains no quadrics. Hence the preimage of (\ref{eq:wedge2})
in $\CC[x]$ is unique.  The final statement follows from part (1)
in the Corollary in \cite[\S 1]{OR}. The next example  illustrates that statement.
 \end{proof}
 
\begin{example} \label{ex:a1b2}
\rm
Let $n=2$ and $k=d$ even, so $X$ is the rational normal curve in $\PP^d$.
Consider the polynomial $(a_1 b_2 - a_2 b_1)^k $.
For $k=2$, its preimage in $\CC[x]$ is  $x_0x_2 - x_1^2$. This
does not vanish on $\tau(X) = \PP^2$. For $k=4$, the preimage is
$ x_0 x_4 - 4 x_1 x_3 + 3 x_2^2$. This vanishes on $\tau(X)$.
\end{example}

For any pair $(\mu,\nu)$ as in (\ref{eq:munu}),
we write $f_{(\mu,\nu)}$ for the unique preimage 
of (\ref{eq:wedge2}) under the map
$\CC[x] \rightarrow \CC[a,b]$. This is well-defined by Proposition \ref{prop:uniquepreimage}.
The polynomial $f_{(\mu,\nu)}$ is easily computable by solving a linear system of equations.
For instance, two $x$-polynomials in Example \ref{ex:a1b2} are $f_{(11,22)}$ and
$f_{(1111,2222)}$. Or, using tableaux, we might write $f_{11 \atop 22}$ and
$f_{1111 \atop 2222}$.

\begin{corollary} \label{eq:quadrics}
A basis for the quadrics that vanish on the tangential variety $\tau(X)$
of the Veronese variety $X $ consists of the $f_{(\mu,\nu)}$ that
are indexed by the SSYT of shape $\lambda = (2d-k,k)$ where
$k \in \{4,5,\ldots,d\}$ is even. Their number is obtained by
summing (\ref{eq:HLF}) over those~$k$.
\end{corollary}

There are no quadrics that vanish on $\tau(X)$ when $d \leq 3$.
For $d \geq 4$ we have constructed an explicit basis for that space of quadrics.
The dimensions of this space is given in Table~\ref{tab:repdim}.

\begin{table}[h]
$$
\begin{matrix}   & \,\,d& 4 & 5 & 6 & 7 & 8 & 9 & 10 \vspace {-0.15in} \\
n & &   &     &    &      &      &      & \\
2 && 1 &  3 & 6 & 10 & 15 & 21 & 28 \\
3 && 15 & 60 & 153 & 315 & 570 & 945 & 1470 \\
4 && 105 & 540 & 1711 & 4270 & 9190 & 17850 & 32130 \\
5 && 490 & 3150 & 12145 & 36155 & 91395 & 205905 & 425425 \\
\end{matrix}
$$
\vspace{-0.1in}
\caption{\vspace{-0.1in}
\label{tab:repdim} Dimension of the space of quadrics vanishing on $\tau\bigl(\nu_d(\PP^{n-1})\bigr)$}
\end{table}

\begin{proposition} \label{prop:landsberg}
Fix a Veronese variety $X = \nu_d(\PP^{n-1})$ with $d \geq 4$.
The tangential variety $\tau(X)$ is defined, as a subvariety of the secant variety $\sigma(X)$,
by the quadrics in Corollary~\ref{eq:quadrics}.
\end{proposition}

\begin{proof}
This is proved in Landsberg's book on tensors, namely in \cite[Theorem 8.1.4.1]{Lan}.
\end{proof}

\begin{example} \label{eq:ternaryquartics} \rm
For ternary quartics ($n=3, d=4$), we consider the
$6 \times 6$ Hankel matrix
$$ H \quad = \quad 
\begin{small}
\begin{pmatrix} 
x_{400} & x_{220} & x_{202} & x_{310} & x_{301} & x_{211}  \\ 
 x_{220} & x_{040} & x_{022} & x_{130} & x_{121} & x_{031} \\ 
 x_{202} & x_{022} & x_{004} & x_{112} & x_{103} & x_{013} \\ 
 x_{310} & x_{130} & x_{112} & x_{220} & x_{211} & x_{121} \\ 
 x_{301} & x_{121} & x_{103} & x_{211} & x_{202} & x_{112} \\ 
 x_{211} & x_{031} & x_{013} & x_{121} & x_{112} & x_{022}   
 \end{pmatrix}. 
 \end{small}
$$ 
The Veronese surface $X \subset \PP^{14}$ is defined by the
$2 \times 2$-minors of $H$. The $5$-dimensional secant variety $\sigma(X)$ is defined
by the $3 \times 3$-minors of $H$.
The tangential variety $\tau(X)$ is the codimension one subvariety of $\sigma(X)$
defined by the vanishing of the following $15$ quadrics:
$f_{(1111, 2222)}$,
$f_{(1111, 2223)}$,
$f_{(1111, 2233)}$,
$f_{(1111, 2333)}$,
$f_{(1111, 3333)}$,
$f_{(1112, 2223)}$,
$f_{(1112, 2233)}$,
$f_{(1112, 2333)}$,
$f_{(1112, 3333)}$,
$f_{(1122, 2233)}$,
$f_{(1122, 2333)}$,
$f_{(1122, 3333)}$,
$f_{(1222, 2333)}$,
$f_{(1222, 3333)}$,
$f_{(2222, 3333)}$.
Each of these symbols translates into a product of $k=4$ factors as in (\ref{eq:wedge2}),
and from this we recover the quadric. For instance,
$ f_{(1111,2222)} = (a_1 b_2 - a_2 b_1)^4 =  x_{400} x_{040}-4 x_{310} x_{130}+3 x_{220}^2$
and $ f_{(1112,2333)} = 
(a_1 b_2 - a_2 b_1) (a_1 b_3 - a_3 b_1)^2 (a_2 b_3 - a_3 b_2) =
x_{310}x_{013}-x_{301}x_{022}-x_{220}x_{103}-x_{211} x_{112}+2 x_{202} x_{121}$.
\end{example}

\begin{remark} \rm
The quadratic polynomials $f_{\mu,\nu}$ that cut out $\tau(X)$
do not contribute to the semi-algebraic description of the real rank two locus $\rho(X)$.
Unlike the hyperdeterminants in Theorem \ref{thm:subhyper},
they do not give valid non-trivial inequalities for $\rho(X)$.
For instance, in Example~\ref{eq:ternaryquartics}, the polynomial
$f_{(1111,2222)}$ is non-negative on $\sigma(X)_\RR$ while  $f_{(1112,2333)}$ 
changes sign on $\rho(X)$. 
Here, $\rho(X)$ is defined in $\sigma(X)_\RR$ by nine quartic inequalities;
cf.~(\ref{eq:theirnumbersym2}) and (\ref{eq:Q_0}).
\end{remark}

\smallskip

For the remainder of this paper we set $n=2$, so we consider
 symmetric $2 {\times} 2 {\times} \cdots {\times} 2$-tensors.
  These tensors form a projective space  $\PP^d$,  namely the space of  binary forms
  \begin{equation}
  \label{eq:binaryform}
  f \quad = \quad \sum_{i=0}^d x_i \binom{d}{i} s^{d-i} t^i.
  \end{equation}
To describe the relevant varieties, we use the following
{\em Hankel matrix} of format $3 \times (d-1)$:
\begin{equation}
\label{eq:hankelmatrix}
 H \quad = \quad
\begin{pmatrix} x_0 & x_1 & x_2 & \cdots & x_{d-2} \\ 
x_1 & x_2 & x_3 & \cdots & x_{d-1} \\ 
x_2 & x_3 & x_4 & \cdots & x_d \end{pmatrix}. 
\end{equation}
Our three varieties of interest satisfy the inclusions $X \subset \tau(X) \subset \sigma(X)$
in $\PP^d$. They are
\begin{itemize}
\item $X \,=\, \{ {\rm rank}(H) \leq 1\} \,=\, \hbox{the rational normal curve in $\PP^d$} 
\,=\, \{\hbox{binary forms $\ell^d$}\}$; \vspace{-0.1in}
\item $\tau(X) \,=\,\,\hbox{points on tangent lines of the curve $X$} 
\,=\, \{\hbox{binary forms $\ell_1^{d-1} \ell_2 $}\}$; \vspace{-0.1in}
\item $\sigma(X) \,=\, \{ {\rm rank}(H) \leq 2\} \, = \,
\hbox{points on secant lines of $X$}
\,=\, \{\hbox{binary forms $\ell_1^d + \ell_2^d $}\}$.
\end{itemize}
These projective varieties have dimensions $1,2$ and $3$.
Their defining equations are as follows.

\begin{corollary}\label{tauandsigma}
The prime ideals of $X$ and $\sigma(X)$ 
are respectively generated by the $2 \times 2$-minors 
and the $3 \times 3$-minors of the Hankel matrix $H$ in (\ref{eq:hankelmatrix}).
The prime ideal of the tangential variety $\tau(X)$ is minimally
generated by the quartic $D$ if $d=3$, by the cubic ${\rm det}(H)$ 
and the quadric $Q = x_0 x_4 - 4 x_1 x_3 + 3 x_2^2$ if $d=4$,
and by $\binom{d-2}{2}$ linearly independent quadrics if $d \geq 5$. 
\end{corollary}

\begin{proof}
The equations for $X$ and $\sigma(X)$ are classical and found
in many sources, such as \cite{Lan}.
The ideal of $\tau(X)$ is derived from the description in Theorem~\ref{thm:oedingraicu}
and Corollary~\ref{eq:quadrics}.
\end{proof}

The real rank two locus $\rho(X)$ is a $3$-dimensional
semi-algebraic set. It consists of
binary forms $\ell_1^d + \ell_2^d$ where $\ell_1,\ell_2 $ are real.
Its algebraic boundary is $\tau(X)$.
Theorem \ref{thm:subhyper} implies:  

\begin{corollary} 
\label{cor:among}
The real rank two locus $\rho(X)$ is the subset of $\,\PP^d_\RR$ that is defined
by the vanishing of the $3 \times 3$-minors of $H$ in
(\ref{eq:hankelmatrix}) together with the following $d-2$ quartic inequalities:
\begin{equation}
\label{eq:symmhyperdet}
 \!\! x_i^2 x_{i+3}^2 - 6 x_i x_{i+1} x_{i+2} x_{i+3} 
  - 3 x_{i+1}^2 x_{i+2}^2+4x_{i+1}^3 x_{i+3}+4x_i x_{i+2}^3 \,\geq \,0
  \quad \,\,
  \hbox{for $i = 0,1,\ldots,d-3$.}
  \end{equation} 
\end{corollary}

\begin{proof}
We regard $f$ as a  $2 \times 2 \times \cdots \times 2$-tensor with $d$ factors,
and we apply (\ref{eq:theirnumbersym2}) and (\ref{eq:Q_0}).
\end{proof}

We conclude by examining the real rank two loci for
binary quartics and binary quintics.

\begin{example} \label{2x2x2x2} \rm
Let $d=4$. We examine the geography of the real hypersurface  $\sigma(X)_\RR$ in $\PP^4_\RR$.
 It decomposes into three semi-algebraic strata.
 Up to closure, these strata are:  the set
$ \sigma(X)^{++0} = \{\ell_1^4 + \ell_2^4\}$ 
of semi-definite real rank two quartics;
the set $ \sigma(X)^{+-0} = \{\ell_1^4 - \ell_2^4\}$
of indefinite real rank two quartics;
the set $\sigma(X)^{{\rm cpx}}$ of quartics
of real rank three and complex rank two. The set $\sigma(X)^{{\rm cpx}}$ is parametrized by terms $\ell^4 + \bar{\ell}^4$, where $\ell$ is a complex linear form and $\bar{\ell}$ its complex conjugate .
 All three strata intersect in the  curve $X_\RR$ of
 rank one quartics.
 
 We examine the points on the boundary  $\partial(\rho(X))$. Limits of points in $\sigma(X)^{++0}$ have real rank one or two, because cancelation between the two positive summands in $l_1^4 + l_2^4$ cannot occur.
 Hence, points in $\partial ( \rho(X) )\backslash X $ must be in the closure of $\sigma(X)^{+-0}$. A typical example~is
$$ \begin{matrix}
s^3 t & = & \lim_{\epsilon \rightarrow 0} \frac{1}{4 \epsilon} \bigl( (s + \epsilon t)^4 - s^4 \bigr) 
& = & \lim_{\epsilon \rightarrow 0} ( s^3 t - \epsilon^2 s t^3 ) . 
\end{matrix}
$$
The first limit approaches $s^3 t$ from within $\sigma(X)^{+-0}$.
The second limit approaches from within  the real rank three locus $\sigma(X)^{{\rm cpx}}$.
To see this, we express it in the form $\ell^4 + \bar{\ell}^4$. Setting $i = \sqrt{-1}$ we have
$$ 
s^3 t - \epsilon^2 s t^3 \quad = \quad
\frac{1}{8 \epsilon i} \bigl( (s + \epsilon i t)^4 - (s- \epsilon i t)^4 \bigr). $$
Since the above decomposition is unique, the tensor is in $\sigma(X)^{{\rm cpx}}$.

 The real rank two locus $\rho(X)$  is defined by 
the equation  ${\rm det}(H) = 0$  and two inequalities $ D_0 \geq 0, D_1 \geq 0$.
Here $D_0$ is the quartic in (\ref{eq:Q_0}) and $D_1$ is obtained by 
replacing $x_i \mapsto x_{i+1}$
for all unknowns.
The variety $ V \bigl( {\rm det}(H) , D_0 , D_1 \bigr)$ has two irreducible
components, namely the line $V(x_1,x_2,x_3)$ and the
surface $\tau(X) = V\bigl(\,{\rm det}(H),\, 3x_2^2-4 x_1 x_3+x_0 x_4\,\bigr)$.
Hence, the real rank two boundary is not obtained by  
setting the inequalities in Corollary \ref{cor:among} to zero.
Note that the rank two tensor $T$ in  (\ref{eq:x^4+y^4})
is symmetric and lies in $V(x_1,x_2,x_3)$.
 \end{example}
 
 \begin{example}\label{dis5} \rm
Let $d=5$. Then $\rho(X)$ is defined by ${\rm rank}(H) \leq 2$ and 
 three inequalities $D_0,D_1,D_2 \geq 0$.
The ideal of the tangential surface $\tau(X)$ is generated by three quadrics
\begin{equation}
\label{eq:threePs}
 Q_0 = 3x_2^2-4x_1x_3+x_0x_4, \,\,
 Q_1 = 2 x_2 x_3-3 x_1 x_4+x_0 x_5,\,\,
Q_2 =  3 x_3^2-4 x_2 x_4+x_1 x_5. 
\end{equation}
It turns out that one inequality suffices to define the real rank two locus
inside the rank two locus. Namely, $\rho(X)$ is the  set of binary quintics given by
${\rm rank}(H) \leq 2$ and $Q_1^2-4 Q_0 Q_2 \geq 0$.
\end{example}

\medskip

\subsubsection*{Acknowledgements}
We are grateful to Luke Oeding and Kristian Ranestad for helpful comments on this project.
Anna Seigal received partial funding from the Pachter Lab and NIH grant R01HG008164.
Bernd Sturmfels was partially supported by
the US National Science Foundation (DMS-1419018)
 and the Einstein Foundation Berlin.
 The article was completed when both authors visited the
 Max-Planck Institute for Mathematics in the Sciences,
 Leipzig, Germany.
\bigskip

\begin{small}

\end{small}

\bigskip \medskip \bigskip

\noindent
\footnotesize {\bf Authors' addresses:}

\smallskip

\noindent Anna Seigal,
University of California, Berkeley, USA,
{\tt seigal@berkeley.edu}

\noindent Bernd Sturmfels, 
University of California, Berkeley, USA,
{\tt bernd@berkeley.edu}

\end{document}